\documentclass{amsart}

\usepackage{amsmath,amscd}
\usepackage{amssymb}

\usepackage{times}

\usepackage[T1]{fontenc}
\usepackage{graphicx}
\usepackage[svgnames]{xcolor}
\usepackage{framed}

\usepackage[all]{xy}

\author{Liran Shaul}

\address{Fakult\"at f\"ur Mathematik\\ 
Universit\"at Bielefeld\\ 
33501 Bielefeld\\ 
Germany.}

\email{LShaul@math.uni-bielefeld.de}
\newtheorem{thm}[equation]{Theorem}
\newtheorem{cor}[equation]{Corollary}
\newtheorem{prop}[equation]{Proposition}
\newtheorem{lem}[equation]{Lemma}

\theoremstyle{definition}

\newtheorem{rem}[equation]{Remark}

\newtheorem{thmx}{Theorem}

\newcommand{\iso}{\xrightarrow{\simeq}}
\newcommand{\inj}{\hookrightarrow}

\newcommand{\opn}{\operatorname}
\newcommand{\cat}[1]{\operatorname{\mathsf{#1}}}

\newcommand{\mfrak}[1]{\mathfrak{#1}}

\newcommand{\mrm}[1]{\mathrm{#1}}
\newcommand{\mbb}[1]{\mathbb{#1}}

\newcommand{\K}{\mbb{K} \hspace{0.05em}}

\renewcommand{\a}{\mfrak{a}}

\newcommand{\injdim}{\operatorname{inj\,dim}}
\newcommand{\flatdim}{\operatorname{fl\,dim}}
\newcommand{\projdim}{\operatorname{proj\,dim}}
\newcommand{\amp}{\operatorname{amp}}
\newcommand{\Ext}{\operatorname{Ext}}
\newcommand{\Tor}{\operatorname{Tor}}
\newcommand{\RHom}{\mrm{R}\opn{Hom}}
\newcommand{\RGamma}{\mrm{R}\Gamma}

\numberwithin{equation}{section} 

\thanks{{\em Mathematics Subject Classification} 2010:
13D05, 13D45, 13B35, 16E45\\
Keywords: Local cohomology, Derived completion, Homological dimension, Commutative DG-rings}

\begin{document}

\title[Homological dimensions of local (co)homology]{Homological dimensions of local (co)homology over commutative DG-rings}

\begin{abstract}
Let $A$ be a commutative noetherian ring,
let $\a\subseteq A$ be an ideal,
and let $I$ be an injective $A$-module.
A basic result in the structure theory of injective modules states that
the $A$-module $\Gamma_{\a}(I)$ consisting of $\a$-torsion elements is also an injective $A$-module. 
Recently, de Jong proved a dual result: If $F$ is a flat $A$-module,
then the $\a$-adic completion of $F$ is also a flat $A$-module.
In this paper we generalize these facts to commutative noetherian DG-rings:
let $A$ be a commutative non-positive DG-ring such that $\mrm{H}^0(A)$ is a noetherian ring,
and for each $i<0$, the $\mrm{H}^0(A)$-module $\mrm{H}^i(A)$ is finitely generated.
Given an ideal $\bar{\a} \subseteq \mrm{H}^0(A)$, 
we show that the local cohomology functor $\mrm{R}\Gamma_{\bar{\a}}$ associated to $\bar{\a}$ does not increase injective dimension.
Dually, the derived $\bar{\a}$-adic completion functor $\mrm{L}\Lambda_{\bar{\a}}$ does not increase flat dimension.
\end{abstract}
\maketitle

\section{Introduction}

\subsection{Torsion of injective modules}

Let $A$ be a commutative noetherian ring,
and $\a\subseteq A$ an ideal.
The $\a$-torsion functor is the functor $\Gamma_{\a}(M) := \varinjlim \opn{Hom}_A(A/\a^n,M)$
which maps an $A$-module $M$ to its submodule consisting $\a$-torsion elements.
An important consequence of Matlis' structure theory of injective modules is the following result:

\begin{thmx}\label{thm:classical}
Let $A$ be a commutative noetherian ring,
let $\a\subseteq A$ be an ideal,
and let $J$ be an injective $A$-module.
Then $\Gamma_{\a}(J)$ is also an injective $A$-module.
\end{thmx}

This basic result depends on $A$ being noetherian.
It is false in general if $A$ is not noetherian,
even if $\a$ is finitely generated (see Remark \ref{exa:fails} below).

We denote by $\cat{D}(A)$ the unbounded derived category of $A$-modules,
and by $\cat{D}^{\mrm{b}}(A)$ its full triangulated subcategory of complexes with bounded cohomology.
For $M \in \cat{D}(A)$, its injective dimension, denoted by $\injdim_A(M)$ was defined in \cite[Definition 2.1.I]{AF}.
The functor $\Gamma_{\a}$ has a right derived functor
\[
\mrm{R}\Gamma_{\a} : \cat{D}(A) \to \cat{D}(A)
\]
It is calculated using K-injective resolutions.
An immediate corollary of Theorem \ref{thm:classical} is the following result:

\begin{cor}\label{cor:inj}
Let $A$ be a commutative noetherian ring,
and let $\a\subseteq A$ be an ideal.
Then for any $M \in \cat{D}^{\mrm{b}}(A)$, we have the following inequality:
\[
\injdim_A(\mrm{R}\Gamma_{\a}(M)) \le \injdim_A(M).
\]
\end{cor}

We now switch our attention to commutative non-positive noetherian DG-rings.
The definition of a commutative DG-ring is recalled in Section \ref{sec:cdg} below.
Just like commutative rings represent affine schemes, 
commutative DG-rings represent affine derived schemes.
Given a commutative DG-ring $A$,
the fact that $A^i = 0$ for $i>0$ implies that
$\mrm{H}^0(A)$ is a commutative ring.
Following \cite{Ye2}, we set 
\[
\bar{A} := \mrm{H}^0(A).
\]
A commutative DG-ring $A$ is called noetherian if the commutative ring $\bar{A}$ is noetherian,
and $\mrm{H}^i(A)$ is a finitely generated $\bar{A}$-module for every $i<0$.
We will denote by $\cat{D}(A)$ the unbounded derived category of DG-modules over $A$,
and by $\cat{D}^{\mrm{b}}(A)$ its full triangulated subcategory of DG-modules with bounded cohomology.

Commutative noetherian DG-rings arise naturally in algebraic geometry:
If $\K$ is a commutative noetherian ring, 
and $A,B$ are finite type $\K$-algebras then $A\otimes^{\mrm{L}}_{\K} B$ is a commutative noetherian DG-ring,
and if $\K$ is not a field, it is often cannot be represented using ordinary commutative rings.

Let $A$ be a commutative DG-ring,
and let $\bar{\a} \subseteq \bar{A}$ be a finitely generated ideal.
In our recent paper \cite{Sh} we introduced a triangulated functor
\[
\mrm{R}\Gamma_{\bar{\a}}: \cat{D}(A) \to \cat{D}(A)
\]
called the right derived $\bar{\a}$-torsion or local cohomology at $\bar{\a}$ functor.
When $A$ is a commutative noetherian ring, this functor coincides with the usual local cohomology functor.
The definition of $\mrm{R}\Gamma_{\bar{\a}}$ as well as an explicit construction of it using the telescope complex is recalled in Section \ref{seclc} below.
To any $M \in \cat{D}(A)$ we can associate its injective dimension, denoted by $\injdim_A(M)$. 
The definition is recalled in Section \ref{sec:hdim}. 

The first main result of this paper generalizes Corollary \ref{cor:inj} to commutative noetherian DG-rings:

\begin{thm}\label{thm:mainA}
Let $A$ be a commutative noetherian DG-ring,
and let $\bar{\a} \subseteq \mrm{H}^0(A)$ be an ideal.
Then for any $M \in \cat{D}^{\mrm{b}}(A)$, there is an inequality
\[
\injdim_A \left(\mrm{R}\Gamma_{\bar{\a}}(M)\right) \le \injdim_A (M).
\] 
\end{thm}
We will prove this result in Theorem \ref{thm:main} below.

\subsection{Completion of flat modules}

Given a commutative ring $A$ and a finitely generated ideal $\a\subseteq A$,
the $\a$-adic completion functor is the functor $\Lambda_{\a}(M) := \varprojlim A/{\a}^n \otimes_A M$.
This construction is dual to the $\a$-torsion functor.
This is best demonstrated by the Greenlees-May duality(\cite{GM}), which states that if $A$ is noetherian then the derived functors
of $\a$-torsion and $\a$-adic completion are adjoint to each other.
This intimate connection between $\a$-torsion and $\a$-adic completion raises the question: 
is there a dual result to Theorem \ref{thm:classical} for the $\a$-adic completion functor?
We have the following result:

\begin{thmx}\label{thm:b}
Let $A$ be a commutative noetherian ring,
let $\a\subseteq A$ be an ideal,
and let $F$ be a flat $A$-module.
Then $\Lambda_{\a}(F)$ is also a flat $A$-module.
\end{thmx}

Theorem \ref{thm:classical} is a classical result which has been known for many decades.
In contrast, Theorem \ref{thm:b} was first proved by de Jong in 2013 (see \cite[Tag 0AGW]{SP}).
If $F$ is a finitely generated $A$-module then this result is trivial.
Prior to de Jong's proof, the best known result was by Enochs,
who showed in \cite{En} that this holds under the additional assumption that $A$ has finite Krull dimension.
After de Jong's proof, other proofs were given by Gabber and Ramero, and by Yekutieli (\cite[Theorem 0.1]{Ye3}).

For a complex of $A$-modules $M$, its flat dimension,
denoted by $\flatdim_A(M)$ was defined in \cite[Definition 2.1.F]{AF}.
The functor $\Lambda_{\a}$ has a left derived functor
\[
\mrm{L}\Lambda_{\a}:\cat{D}(A) \to \cat{D}(A).
\]
It is calculated using K-flat resolutions.
The following is am immediate corollary of Theorem \ref{thm:b}:

\begin{cor}\label{cor:flat}
Let $A$ be a commutative noetherian ring,
and let $\a\subseteq A$ be an ideal.
Then for any $M \in \cat{D}^{\mrm{b}}(A)$, we have the following inequality:
\[
\flatdim_A \left( \mrm{L}\Lambda_{\a}(M) \right) \le \flatdim_A (M).
\]
\end{cor}

Now let $A$ be a commutative DG-ring.
Given a finitely generated ideal $\bar{\a} \subseteq \bar{A}$,
there is a triangulated functor 
\[
\mrm{L}\Lambda_{\bar{\a}} : \cat{D}(A) \to \cat{D}(A)
\]
called the derived $\bar{\a}$-adic completion functor.
Its definition is recalled in Section \ref{seclc} below.
If $A$ is a commutative noetherian ring, it coincides with the usual left derived functor of adic completion.
To any $M \in \cat{D}(A)$ we can associate its flat dimension, denoted by $\flatdim_A(M)$. 
The definition is recalled in Section \ref{sec:hdim}. 

The second main result of this paper generalizes Corollary \ref{cor:flat} to commutative noetherian DG-rings:

\begin{thm}\label{thm:mainB}
Let $A$ be a commutative noetherian DG-ring,
and let $\bar{\a} \subseteq \mrm{H}^0(A)$ be an ideal.
Then for any $M \in \cat{D}^{\mrm{b}}(A)$, there is an inequality
\[
\flatdim_A \left(\mrm{L}\Lambda_{\bar{\a}}(M)\right) \le \flatdim_A (M).
\] 
\end{thm}

This will be proved in Theorem \ref{thm:main2} below.

\section{Preliminaries}

In this section we recall some basic facts concerning commutative DG-rings.

\subsection{Commutative DG-rings}\label{sec:cdg}

A DG-ring $A$ is a $\mathbb{Z}$-graded ring
\[
A = \bigoplus_{n=-\infty}^{\infty} A^n
\]
together with an additive differential $d:A \to A$ of degree $+1$,
such that $d\circ d = 0$,
and such that the Leibniz rule holds: $d(a\cdot b) = d(a) \cdot b + (-1)^i \cdot a \cdot d(b)$ for all $a \in A^i$ and $b \in A^j$.
A DG-ring $A$ is called non-positive if $A^i = 0$ for all $i > 0$.
We say that $A$ is commutative if $b\cdot a = (-1)^{i\cdot j} \cdot a \cdot b$ for all $a \in A^i$ and $b \in A^j$,
and moreover $a\cdot a = 0$ if $i$ is odd.

In this paper, \textbf{all DG-rings are assumed to be commutative and non-positive}.
Given a commutative DG-ring $A$, 
a DG-module over it is graded $A$-module $M$ with a differential $d:M \to M$ of degree $+1$
satisfying a graded Leibniz rule. The category of all DG-modules is denoted by $\opn{DGMod}(A)$.
Inverting quasi-isomorphisms in it, we obtain the derived category of DG-modules over $A$, denoted by $\cat{D}(A)$.

If $A$ is a commutative DG-ring, recall from the introduction that we denote by $\bar{A}$ the commutative ring $\mrm{H}^0(A)$.
The DG-ring $A$ is called noetherian if $\bar{A}$ is noetherian and $\mrm{H}^i(A)$ is a finitely generated $\bar{A}$-module for all $i<0$.

\subsection{Local (co)homology over commutative DG-rings and the telescope complex}\label{seclc}

Let $A$ be a commutative DG-ring,
and let $\bar{\a} \subseteq \bar{A}$ be a finitely generated ideal.
The category of derived $\bar{\a}$-torsion DG-modules,
denoted by $\cat{D}_{\bar{\a}-\opn{tor}}(A)$,
is the full triangulated subcategory of $\cat{D}(A)$,
consisting of DG-modules $M$,
such that for all $n\in \mathbb{Z}$,
the $\bar{A}$-module $\mrm{H}^n(M)$ is $\bar{\a}$-torsion.

According to \cite[Theorem 2.13(1)]{Sh},
the inclusion functor $i_{\bar{\a}}:\cat{D}_{\bar{\a}-\opn{tor}}(A) \inj \cat{D}(A)$ has a right adjoint $F_{\bar{\a}}:\cat{D}(A) \to \cat{D}_{\bar{\a}-\opn{tor}}(A)$.
The composition $i_{\bar{\a}} \circ F_{\bar{\a}}: \cat{D}(A) \to \cat{D}(A)$ is denoted by 
$\mrm{R}\Gamma_{\bar{\a}}:\cat{D}(A) \to \cat{D}(A)$ and called the local cohomology functor of $A$ with respect to $\bar{\a}$.

By \cite[Theorem 2.13(2)]{Sh}, 
the functor $\mrm{R}\Gamma_{\bar{\a}}$ has a left adjoint 
$\mrm{L}\Lambda_{\bar{\a}}:\cat{D}(A) \to \cat{D}(A)$ called the derived completion (or local homology) functor of $A$ with respect to $\bar{\a}$.

Below, we will give explicit formulas for the functors $\mrm{R}\Gamma_{\bar{\a}}, \mrm{L}\Lambda_{\bar{\a}}$.

\begin{rem}\label{rem:noetherianring}
If $A$ is an ordinary commutative noetherian ring,
by \cite[Theorem 7.12]{PSY1},
these constructions of local cohomology and derived completion with respect to $\bar{\a}$ coincide with the usual right derived functor of the functor $\Gamma_{\bar{\a}}$ 
and left derived functor of $\Lambda_{\bar{\a}}$.
\end{rem}

To give an explicit formula for the local cohomology $\mrm{R}\Gamma_{\bar{\a}}$ and derived completion $\mrm{L}\Lambda_{\bar{\a}}$,
we recall the construction of the telescope complex from \cite[Section 5]{PSY1}:
given a commutative ring $A$ and some $a \in A$,
the telescope complex $\opn{Tel}(A;a)$ is the cochain complex 
\[
0 \to \bigoplus_{i=0}^\infty A \xrightarrow{d} \bigoplus_{i=0}^\infty A \to 0
\]
with non-zero components in degrees $0,1$.
Letting $\{\delta_i\mid i\ge 0\}$ be the basis of the countably generated free A-module 
$\bigoplus_{i=0}^\infty A$, the differential $d$ is defined by 
\[
d(\delta_i) = \begin{cases}
               \delta_0 & \text{if $i=0$}\\
               \delta_{i-1}-a\cdot \delta_i & \text{if $i\ge 1$}
              \end{cases}
\]
Given a finite sequence $\mathbf{a}=(a_1,\dots,a_n)$ of elements of $A$,
the telescope complex associated to $\mathbf{a}$ is the complex
\[
\opn{Tel}(A;\mathbf{a}) := \opn{Tel}(A;a_1) \otimes_A \opn{Tel}(A;a_2) \otimes_A \dots \otimes_A \opn{Tel}(A;a_n). 
\]
This is a bounded complex of free $A$-modules.
The telescope complex has the following base change property: 
if $f:A \to B$ is a homomorphism between commutative rings,
and $\mathbf{b} = (f(a_1),f(a_2),\dots,f(a_n))$,
there is an isomorphism of complexes of $B$-modules $\opn{Tel}(A;\mathbf{a})\otimes_A B \cong \opn{Tel}(B;\mathbf{b})$.

Let $A$ be a commutative DG-ring,
and let $\bar{\a} \subseteq \bar{A}$ be a finitely generated ideal.
Let $\bar{\mathbf{a}} = (\bar{a}_1,\dots,\bar{a}_n)$ be a finite sequence of elements in $\bar{A}$ that generates $\bar{\a}$,
and using the surjection $A^0 \to \bar{A}$,
choose some lifts $\mathbf{a}=(a_1,\dots,a_n)$ of $\bar{\mathbf{a}}$ to the commutative ring $A^0$.
By \cite[Proposition 2.4]{Sh} and the base change property of the telescope complex,
there are isomorphisms
\begin{equation}\label{eqn:telescope}
 \mrm{R}\Gamma_{\bar{\a}}(-) \cong \opn{Tel}(A^0;\mathbf{a}) \otimes_{A^0} -
\end{equation}
and
\begin{equation}\label{eqn:telescope2}
 \mrm{L}\Lambda_{\bar{\a}}(-) \cong \opn{Hom}_{A^0}(\opn{Tel}(A^0;\mathbf{a}),-)
\end{equation}
of functors $\cat{D}(A) \to \cat{D}(A)$.

\subsection{Homological dimensions over commutative DG-rings}\label{sec:hdim}

Let $A$ be a commutative DG-ring.
Given a DG-module $M$ over $A$,
we set 
\[
\inf(M) := \inf \{ i \in \mathbb{Z} \mid \mrm{H}^i(M) \ne 0 \}, \quad \sup(M) := \sup \{ i \in \mathbb{Z} \mid \mrm{H}^i(M) \ne 0 \}
\]
and $\amp(M) := \sup(M) - \inf(M)$. 

Let $M$ be a DG-module over $A$.
Following \cite[Section 2.I]{AF},
we let the injective dimension of $M$, denoted by $\injdim_A (M)$ to be the number:
\[
\inf \{ n \in \mathbb{Z} \mid \Ext^i_A(N,M) = 0 \text{ for any bounded DG-module } N \text{ and any } i>n-\inf N\}
\]
where $\Ext^i_A(N,M) := \mrm{H}^i(\mrm{R}\opn{Hom}_A(N,M))$.
By \cite[Theorem 2.4.I]{AF}\footnote{Unlike \cite{AF}, in this paper we use a cohomological notation, hence the difference between the formulas.} this coincides with the usual definition in case $A$ is a commutative ring.
Similarly, we let the flat dimension of $M$, denoted by $\flatdim_A (M)$ to be the number:
\[
\inf \{ n \in \mathbb{Z} \mid \Tor^i_A(N,M) = 0 \text{ for any bounded DG-module } N \text{ and any } i>n-\inf N\}
\] 
where $\Tor^i_A(N,M) := \mrm{H}^{-i}(N\otimes^{\mrm{L}}_A M)$.
The projective dimension of $M$, denoted by $\projdim_A(M)$ is defined similarly.

\subsection{Subcategories of $\cat{D}(A)$}

Given a commutative DG-ring $A$,
we denote by $\cat{D}^{-}(A)$ (respectively by $\cat{D}^{+}(A)$)
the full triangulated subcategory of $A$ consisting of DG-modules with bounded above (resp. bounded below) cohomology.
The full triangulated subcategory of DG-modules with bounded cohomology is $\cat{D}^{\mrm{b}}(A) = \cat{D}^{-}(A) \cap \cat{D}^{+}(A)$.
Assume further that $A$ is noetherian. 
In particular, $\bar{A}$ is a noetherian ring.
We say that $M \in \cat{D}(A)$ has finitely generated cohomology if for all $i \in \mathbb{Z}$,
$\mrm{H}^i(M)$ is a finitely generated $\bar{A}$-module.
The full triangulated subcategory of $\cat{D}(A)$ consisting of DG-modules with finitely generated cohomologies is denoted by $\mrm{D}_{\mrm{f}}(A)$.
We let $\cat{D}^{+}_{\mrm{f}}(A) = \cat{D}^{+}(A) \cap \cat{D}_{\mrm{f}}(A)$,
$\cat{D}^{-}_{\mrm{f}}(A) = \cat{D}^{-}(A) \cap \cat{D}_{\mrm{f}}(A)$
and $\cat{D}^{\mrm{b}}_{\mrm{f}}(A) = \cat{D}^{b}(A) \cap \cat{D}_{\mrm{f}}(A)$.

\subsection{The tensor-evaluation morphism}

Let $A$ be a commutative DG-ring.
Given 
\[
M,N,K \in \cat{D}(A),
\]
there is a natural morphism
\[
\eta_{M,N,K}: \mrm{R}\opn{Hom}_A(M,N)\otimes^{\mrm{L}}_A K \to \mrm{R}\opn{Hom}_A(M,N\otimes^{\mrm{L}}_A K) 
\]
in $\cat{D}(A)$, defined as follows:
Let $P \cong M$ be a K-projective resolution,
and let $F \cong K$ be a K-flat resolution.
Then, $\eta_{M,N,K}$ is the composition
\begin{align*}
 \mrm{R}\opn{Hom}_A(M,N)\otimes^{\mrm{L}}_A K \cong \opn{Hom}_A(P,N)\otimes_A F \xrightarrow{\phi}\\
 \opn{Hom}_A(P,N\otimes_A F) \cong \mrm{R}\opn{Hom}_A(M,N\otimes^{\mrm{L}}_A K) 
\end{align*}
where the map $\phi$ is the usual tensor-evaluation morphism (see \cite[Equation (5.6)]{Ye2} for its formula in the DG-case).
The next result generalizes \cite[Proposition 6.7]{Sh1}.

\begin{prop}\label{prop:eval}
Let $A$ be a commutative noetherian DG-ring, 
and let $M,N,K \in \cat{D}(A)$.
Assume one of the following holds:
\begin{enumerate}
 \item $M \in \cat{D}^{-}_{\mrm{f}}(A)$, $N \in \cat{D}^{\mrm{b}}(A)$ and $\flatdim_A(K) < \infty$.
 \item $\projdim_A(M) < \infty$, $N \in \cat{D}^{\mrm{b}}(A)$ and $K \in \mrm{D}^{-}_{\mrm{f}}(A)$.
\end{enumerate}
Then the morphism
\[
\eta_{M,N,K}: \mrm{R}\opn{Hom}_A(M,N)\otimes^{\mrm{L}}_A K \to \mrm{R}\opn{Hom}_A(M,N\otimes^{\mrm{L}}_A K) 
\]
is an isomorphism in $\cat{D}(A)$.
\end{prop}
\begin{proof}
\begin{enumerate}
 \item Fixing such $N,K$, we have a natural morphism 
 \[
  \zeta_M: \mrm{R}\opn{Hom}_A(M,N)\otimes^{\mrm{L}}_A K \to \mrm{R}\opn{Hom}_A(M,N\otimes^{\mrm{L}}_A K) 
 \]
 These assumptions on $N,K$ ensure that the functors 
 $\mrm{R}\opn{Hom}_A(-,N)\otimes^{\mrm{L}}_A K$ and $\mrm{R}\opn{Hom}_A(-,N\otimes^{\mrm{L}}_A K)$
 are both contravariant way-out right functors. 
 Clearly, $\zeta_A$ is an isomorphism. 
 Hence, by a DG-version of the lemma on way-out functors (for instance, \cite[Theorem 2.11]{Ye2}),
 we deduce that $\zeta_M$ is an isomorphism for any $M \in \cat{D}^{-}_{\mrm{f}}(A)$.
 
 \item Fixing such $M,N$, we have a natural morphism
 \[
  \iota_K: \mrm{R}\opn{Hom}_A(M,N)\otimes^{\mrm{L}}_A K \to \mrm{R}\opn{Hom}_A(M,N\otimes^{\mrm{L}}_A K).
 \]
 These assumptions on $M,N$ ensure that the functors 
 $\mrm{R}\opn{Hom}_A(M,N)\otimes^{\mrm{L}}_A -$ and $\mrm{R}\opn{Hom}_A(M,N\otimes^{\mrm{L}}_A -)$
 are both covariant way-out left functors, and it is clear that $\iota_A$ is an isomorphism. 
 Hence, by the lemma on way-out functors, $\iota_K$ is an isomorphism for any $K \in \mrm{D}^{-}_{\mrm{f}}(A)$.

\end{enumerate}

\end{proof}

\section{Injective dimension of local cohomology}

In this section we will prove Theorem \ref{thm:mainA}.
We begin by proving some basic results about injective dimension over commutative DG-rings.

\begin{prop}\label{prop:injdim-of-RHom}
Let $A \to B$ be a homomorphism between commutative DG-rings,
and let $M$ be a DG-module over $A$.
Then 
\[
\injdim_B \left(\RHom_A(B,M)\right) \le \injdim_A (M).
\]
\end{prop}
\begin{proof}
This follows from the adjunction
\begin{equation*}
\RHom_B(-,\RHom_A(B,M)) \simeq \RHom_A(-,M)
\end{equation*}
\end{proof}

Over a commutative ring $A$, 
it is well known that one can detect the injective dimension of a complex $M$ by 
checking the vanishing of $\Ext^i_A(N,M)$ for all $A$-modules $N$ (that is, complexes with zero amplitude).
The proof of the next result is based on the same idea over a DG-ring $A$,
together with the observation that a DG-module whose amplitude is zero is isomorphic in the derived category to the shift of an $\bar{A}$-module.

\begin{thm}\label{thm:injDG}
Let $A$ be a commutative DG-ring,
and let $M$ be a DG-module over $A$. 
Then there is an equality
\begin{equation*}
\injdim_A (M) = \injdim_{\mrm{H}^0(A)} \left(\RHom_A(\mrm{H}^0(A),M)\right)
\end{equation*}
\end{thm}
\begin{proof}
Applying Proposition \ref{prop:injdim-of-RHom} to the map $A \to \bar{A}$,
we have that

\begin{equation*}
\injdim_A (M) \ge \injdim_{\bar{A}} \left(\RHom_A(\bar{A},M)\right).
\end{equation*}
To prove the converse, assume $\injdim_{\bar{A}} \left(\RHom_A(\bar{A},M)\right) = n < \infty$.
Let $N$ be a bounded DG-module.
We must show that $\Ext^i_A(N,M) = 0$ for all $i>n-\inf N$.
We will prove this by induction on $\amp N$.
If $\amp N = 0$, 
there is some $m \in \mathbb{Z}$ such that 
\begin{equation*}
N \simeq \mrm{H}^m(N)[-m]
\end{equation*}
It follows by adjunction that
\begin{align*}
& \RHom_A(N,M) \simeq \RHom_A(\mrm{H}^m(N)[-m],M) \\
& \simeq \RHom_{\bar{A}}(\mrm{H}^m(N)[-m],\RHom_A(\bar{A},M))
\end{align*}
so the fact that $\injdim_{\bar{A}}\left(\RHom_A(\bar{A},M)\right) = n$ 
implies that $\Ext^i_A(N,M) = 0$ for all $i>n-\inf N$.

Given $l>0$, assume now that for any bounded DG-module $N$ with $\amp N < l$ we have that
$\Ext^i_A(N,M) = 0$ for all $i>n-\inf N$,
and let $N$ be a bounded DG-module with $\amp N = l$.

According to \cite[Page 299]{FIJ}, 
using truncations, the DG-module $N$ fits into a distinguished triangle

\begin{equation}\label{eqn:triangle1}
N' \to N \to N'' \to N[1]
\end{equation}
in $\cat{D}(A)$,  such that the $\amp N' < l$, $\amp N'' <l$,
and moreover $\inf N' \ge \inf N$ and $\inf N'' \ge \inf N$.
Applying the contravariant triangulated functor $\RHom_A(-,M)$
to the triangle (\ref{eqn:triangle1}), 
we obtain a distinguished triangle

\begin{equation}\label{eqn:triangle2}
\RHom_A(N'',M) \to \RHom_A(N,M) \to \RHom_A(N',M) \to \RHom_A(N'',M)[1]
\end{equation}
The result now follows from the induction hypothesis and the long exact sequence in cohomology
associated to the distinguished triangle (\ref{eqn:triangle2}).

\end{proof}

Before stating the next lemma, we shall need the following terminology:

\begin{rem}\label{rem:term}
If $A$ is a commutative DG-ring,
and $\bar{\a} \subseteq \bar{A}$ is a finitely generated ideal,
we can form the local cohomology functor of $A$ with respect to $\bar{\a}$
and the local cohomology functor of $\bar{A}$ with respect to $\bar{\a}$.
The former is a functor $\cat{D}(A) \to \cat{D}(A)$, 
while the latter is a functor $\cat{D}(\bar{A}) \to \cat{D}(\bar{A})$.

According to our notations, both should be denoted by $\mrm{R}\Gamma_{\bar{\a}}$.
In cases where there will be such ambiguity, we solve it by using the notation
\[
\mrm{R}\Gamma_{\bar{\a}}^{\bar{A}} : \cat{D}(\bar{A}) \to \cat{D}(\bar{A})
\]
for the local cohomology functor of $\bar{A}$ with respect to $\bar{\a}$,
and the notation
\[
\mrm{R}\Gamma_{\bar{\a}}^A : \cat{D}(A) \to \cat{D}(A)
\]
for the local cohomology functor of $A$ with respect to $\bar{\a}$.
Similarly, we will write
\[
\mrm{L}\Lambda_{\bar{\a}}^{\bar{A}} : \cat{D}(\bar{A}) \to \cat{D}(\bar{A})
\]
for the derived completion functor of $\bar{A}$ with respect to $\bar{\a}$,
and
\[
\mrm{L}\Lambda_{\bar{\a}}^{A} : \cat{D}(A) \to \cat{D}(A)
\]
for the derived completion functor of $A$ with respect to $\bar{\a}$.
\end{rem}

\begin{lem}\label{lem:rgammaofrhom}
Let $A$ be a commutative noetherian DG-ring,
and let $\bar{\a} \subseteq \mrm{H}^0(A)$ be an ideal.
Then for any DG-module $M$ over $A$ with bounded cohomology,
there is a natural isomorphism
\begin{equation*}
\RGamma_{\bar{\a}}^{\bar{A}} \left(\RHom_A(\bar{A},M)\right) \cong \RHom_A\left(\bar{A},\RGamma_{\bar{\a}}^{A}(M)\right)
\end{equation*}
in $\cat{D}(\bar{A})$.
\end{lem}
\begin{proof}
Let $\mathbf{a}$ be a finite sequence of elements of the ring $A^0$ whose image in $\bar{A}$ generates the ideal $\bar{\a}$,
and let $\bar{\mathbf{a}}$ be its image in $\bar{A}$. The latter is a finite sequence of elements of the ring $\bar{A}$.
By (\ref{eqn:telescope}), 
there is a natural isomorphism
\[
\RGamma_{\bar{\a}}^{\bar{A}} \left(\RHom_A(\bar{A},M)\right) \cong
\opn{Tel}(\bar{A};\bar{\mathbf{a}}) \otimes_{\bar{A}} \RHom_A(\bar{A},M).
\]
By the base change property of the telescope complex, 
there is an isomorphism of complexes of $\bar{A}$-modules:
\[
\opn{Tel}(A^0;\mathbf{a}) \otimes_{A^0} \bar{A} \cong \opn{Tel}(\bar{A};\bar{\mathbf{a}}).
\]
This implies that
\[
\opn{Tel}(\bar{A};\bar{\mathbf{a}}) \otimes_{\bar{A}} \RHom_A(\bar{A},M) \cong  \opn{Tel}(A^0;\mathbf{a}) \otimes_{A^0} \RHom_A(\bar{A},M)
\]
Set $T := A \otimes_{A^0} \opn{Tel}(A^0;\mathbf{a})$.
Since $\opn{Tel}(A^0;\mathbf{a})$ is a K-flat complex of finite flat dimension over $A^0$,
it follows that the DG-module $T$ is K-flat over $A$ and that $\flatdim_A(T) < \infty$.
Hence, it holds that
\[
\opn{Tel}(A^0;\mathbf{a}) \otimes_{A^0} \RHom_A(\bar{A},M) \cong T \otimes^{\mrm{L}}_A \RHom_A(\bar{A},M).
\]
Let $M \cong I$ be a K-injective resolution over $A$, 
let $I \otimes_A T \cong J$ be a K-injective resolution over $A$,
and let $A \to B \iso \bar{A}$ be a semi-free DG-algebra resolution of $\bar{A}$ over $A$.
In particular, $B$ is K-projective over $A$.
These resolutions and the naturality of the tensor evaluation morphism induce a commutative diagram in $\opn{DGMod}(A)$:
\[
\xymatrixcolsep{2pc}
\xymatrixrowsep{2pc}
\xymatrix{
\opn{Hom}_A(\bar{A},I)\otimes_A T \ar[r]^{\beta}\ar[d]_{\alpha} & \opn{Hom}_A(\bar{A},I\otimes_A T) \ar[r]^{\varphi}\ar[d]_{\delta} & \opn{Hom}_A(\bar{A},J)\ar[d]^{\psi}\\
\opn{Hom}_A(B,I)\otimes_A T \ar[r]_{\gamma} & \opn{Hom}_A(B,I\otimes_A T) \ar[r]_{\chi} & \opn{Hom}_A(B,J)
}
\]
Since $I$ is K-injective over $A$ and $T$ is K-projective over $A$,
it follows that $\alpha$ is a quasi-isomorphism.
Similarly, since $B$ is K-projective over $A$,
it follows that $\chi$ is a quasi-isomorphism.
K-injectivity of $J$ over $A$ implies that $\psi$ is a quasi-isomorphism.
Finally, since $A$ is noetherian and $T$ has finite flat dimension over $A$,
it follows by Proposition \ref{prop:eval}(1) that $\gamma$ is a quasi-isomorphism.
Hence, by the 2-out-of-3 property, the $\bar{A}$-linear map $\varphi \circ \beta$ is also a quasi-isomorphism.
Hence, there are natural isomorphisms
\begin{align*}
\mrm{R}\opn{Hom}_A(\bar{A},M) \otimes_A T \cong \opn{Hom}_A(\bar{A},I)\otimes_A T \cong\\
\opn{Hom}_A(\bar{A},J) \cong \mrm{R}\opn{Hom}_A(\bar{A},M\otimes_A T)
\end{align*}
in $\cat{D}(\bar{A})$.
Finally, note that 
\[
M\otimes_A T = M\otimes_A (A \otimes_{A^0} \opn{Tel}(A^0;\mathbf{a})) \cong M\otimes_{A^0} \opn{Tel}(A^0;\mathbf{a}) \cong \RGamma_{\bar{\a}}^{A}(M)
\]
Combining all the above natural isomorphisms gives the required result.
\end{proof}

We now prove the first main result of this paper.

\begin{thm}\label{thm:main}
Let $A$ be a commutative noetherian DG-ring,
and let $\bar{\a} \subseteq \mrm{H}^0(A)$ be an ideal.
Then for any $M \in \cat{D}^{\mrm{b}}(A)$, there is an inequality
\[
\injdim_A \left(\mrm{R}\Gamma_{\bar{\a}}(M)\right) \le \injdim_A (M).
\] 
\end{thm}
\begin{proof}
According to Theorem \ref{thm:injDG}, 
we have that
\begin{align*}
\injdim_A \left(\RGamma_{\bar{\a}} (M)\right) = \injdim_A \left(\RGamma^A_{\bar{\a}} (M)\right) =\\
\injdim_{\bar{A}} \left(\RHom_A(\bar{A},\RGamma^A_{\bar{\a}} (M))\right).
\end{align*}
Using Lemma \ref{lem:rgammaofrhom}, 
we have that
\begin{align*}
 \injdim_{\bar{A}} \left(\RHom_A(\bar{A},\RGamma^A_{\bar{\a}} (M))\right) =\\
 \injdim_{\bar{A}} \left(\RGamma^{\bar{A}}_{\bar{\a}} (\RHom_A(\bar{A}, M))\right)
\end{align*}
and by Corollary \ref{cor:inj} we obtain that
\begin{align*}
\injdim_{\bar{A}} \left(\RGamma^{\bar{A}}_{\bar{\a}} (\RHom_A(\bar{A}, M))\right) \le\\
\injdim_{\bar{A}} \left(\RHom_A(\bar{A}, M)\right) = \injdim_A (M)
\end{align*}
where the last equality follows from Theorem \ref{thm:injDG}.
Combining all of the above, we obtain that
\[
\injdim_A \left(\RGamma_{\bar{\a}} (M)\right) \le \injdim_A (M)
\]
as claimed.

\end{proof}

\begin{rem}\label{exa:fails}
Given a commutative ring $A$, 
and a finitely generated ideal $\a\subseteq{A}$,
unlike Remark \ref{rem:noetherianring}, 
if $A$ is non-noetherian
in general the right derived functor of the $\a$-torsion functor $\Gamma_{\a}$
might be different from the local cohomology functor with respect to $\a$ from Section \ref{seclc}.
However, if $\a$ satisfies a technical condition called weak proregularity (see \cite[Definition 4.21]{PSY1}),
these functors coincide.
In \cite[Proposition 3.1]{QR},
there is an example of a commutative ring $A$,
a finitely generated (in fact, principal) weakly proregular ideal $\a\subseteq A$,
and an injective $A$-module $I$ such that $\Gamma_{\a}(I)$ is not an injective $A$-module.
\end{rem}

\section{Flat dimension of derived completion}

The aim of this section is to prove Theorem \ref{thm:mainB}.
The next result is dual to Theorem \ref{thm:injDG}.

\begin{thm}\label{thm:flatDG}
Let $A$ be a commutative DG-ring, and let $M$ be a DG-module over $A$.
Then there is an equality
\[
\flatdim_A (M)  = \flatdim_{\mrm{H}^0(A)} \left( \mrm{H}^0(A) \otimes^{\mrm{L}}_A M \right)
\]
\end{thm}
\begin{proof}
For any $N \in \cat{D}(\bar{A})$,
the isomorphism
\[
N \otimes^{\mrm{L}}_A M \cong N\otimes^{\mrm{L}}_{\bar{A}} ( \bar{A} \otimes^{\mrm{L}}_A M )
\]
shows that 
\[
\flatdim_A (M) \ge \flatdim_{\bar{A}} \left( \bar{A} \otimes^{\mrm{L}}_A M \right)
\]
To prove the converse, assume that $\flatdim_{\bar{A}} \left( \bar{A} \otimes^{\mrm{L}}_A M \right) = n < \infty$,
and let $N$ be a bounded DG-module.
If $\amp(N) = 0$, there is an isomorphism $N \cong \mrm{H}^m(N)[-m]$ for some $m \in \mathbb{Z}$,
which implies that
\[
M \otimes^{\mrm{L}}_A N \cong \left(M\otimes^{\mrm{L}}_A \bar{A}\right) \otimes^{\mrm{L}}_{\bar{A}} \mrm{H}^m(N)[-m]
\]
Hence, the fact that 
\[
\flatdim_{\bar{A}} \left( \bar{A} \otimes^{\mrm{L}}_A M \right) = n
\]
implies that in this case $\Tor^i_A(N,M) = 0$ for all $i>n-\inf N$.
Now, proceeding by induction on $\amp(N)$,
exactly as in the proof of Theorem \ref{thm:injDG}, 
we obtain the general case for an arbitrary bounded DG-module $N$.
\end{proof}

We will now use again the terminology introduced in Remark \ref{rem:term}.

\begin{lem}\label{lem:LLambda}
Let $A$ be a commutative noetherian DG-ring,
and let $\bar{\a} \subseteq \mrm{H}^0(A)$ be an ideal.
Then for any DG-module $M$ over $A$ with bounded cohomology,
there is a natural isomorphism
\[
\mrm{L}\Lambda^{\bar{A}}_{\bar{\a}} \left(\bar{A} \otimes^{\mrm{L}}_A M \right) \cong    \bar{A} \otimes^{\mrm{L}}_A \left( \mrm{L}\Lambda^{A}_{\bar{\a}}(M) \right)
\]
in $\cat{D}(\bar{A})$.
\end{lem}
\begin{proof}
Let $\mathbf{a}$ be a finite sequence of elements of the ring $A^0$ whose image in $\bar{A}$ generates the ideal $\bar{\a}$,
and let $\bar{\mathbf{a}}$ be its image in $\bar{A}$. The latter is a finite sequence of elements of the ring $\bar{A}$.
By (\ref{eqn:telescope2}), there is a natural isomorphism 
\[
\mrm{L}\Lambda^{\bar{A}}_{\bar{\a}} \left(\bar{A} \otimes^{\mrm{L}}_A M \right) \cong \opn{Hom}_{\bar{A}}(\opn{Tel}(\bar{A};\bar{\mathbf{a}}), \bar{A} \otimes^{\mrm{L}}_A M)
\]
As in the proof of Lemma \ref{lem:rgammaofrhom}, setting $T := \opn{Tel}(A^0;\mathbf{a}) \otimes_{A^0} A$,
observing that $T$ is K-projective over $A$,
and using adjunctions,
this is naturally isomorphic to
\[
\opn{Hom}_A(T, \bar{A} \otimes^{\mrm{L}}_A M).
\]
Let $P \cong M$ be a K-flat resolution over $A$,
let $F \cong \opn{Hom}_A(T,P)$ be a K-flat resolution over $A$,
and let $A \to B \cong \bar{A}$ be a semi-free DG-algebra resolution of $\bar{A}$ over $A$.
We obtain a commutative diagram in $\opn{DGMod}(A)$:
\[
\xymatrixcolsep{3pc}
\xymatrixrowsep{3pc}
\xymatrix{
F \otimes_A \bar{A} \ar[r]^{\beta} & \opn{Hom}_A(T,P)\otimes_A \bar{A} \ar[r]^{\varphi} & \opn{Hom}_A(T,P\otimes_A \bar{A})\\
F \otimes_A B \ar[r]_{\gamma} \ar[u]^{\alpha} & \opn{Hom}_A(T,P)\otimes_A B \ar[r]_{\chi} \ar[u]^{\delta} & \opn{Hom}_A(T,P\otimes_A B) \ar[u]_{\psi}
}
\]
The fact that $F$ is K-flat over $A$ implies that $\alpha$ is a quasi-isomorphism,
while K-flatness of $B$ over $A$ implies that $\gamma$ is a quasi-isomorphism.
The fact that $P$ is K-flat over $A$ and $T$ is K-projective over $A$ implies that $\psi$ is a quasi-isomorphism.
Finally, since $A$ is noetherian and $\projdim_A(T) < \infty$,
by Proposition \ref{prop:eval}(2), the map $\chi$ is a quasi-isomorphism.
It follows by the 2-out-of-3 property that the $\bar{A}$-linear map $\varphi \circ \beta$ is a quasi-isomorphism.
Hence, there are natural isomorphisms
\begin{align*}
\mrm{R}\opn{Hom}_A(T,M)\otimes^{\mrm{L}}_A \bar{A} \cong F\otimes_A \bar{A} \cong\\
\opn{Hom}_A(T,P\otimes_A \bar{A}) \cong \mrm{R}\opn{Hom}_A(T,M\otimes^{\mrm{L}}_A \bar{A})
\end{align*}
in $\cat{D}(\bar{A})$.
The result now follows by combining all these isomorphisms and using the fact that there is a natural isomorphism
\[
\mrm{L}\Lambda^A_{\bar{\a}}(M) \cong \mrm{R}\opn{Hom}_A(T,M)
\]
in $\cat{D}(A)$.
\end{proof}

Here is the second main result of this paper.

\begin{thm}\label{thm:main2}
Let $A$ be a commutative noetherian DG-ring,
and let $\bar{\a} \subseteq \mrm{H}^0(A)$ be an ideal.
Then for any $M \in \cat{D}^{\mrm{b}}(A)$, there is an inequality
\[
\flatdim_A \left(\mrm{L}\Lambda_{\bar{\a}}(M)\right) \le \flatdim_A (M).
\] 
\end{thm}
\begin{proof}
By Theorem \ref{thm:flatDG},
\begin{align*}
\flatdim_A \left(\mrm{L}\Lambda_{\bar{\a}}(M)\right) = \flatdim_A \left(\mrm{L}\Lambda^A_{\bar{\a}}(M)\right) =\\
\flatdim_{\bar{A}} \left( \bar{A}\otimes^{\mrm{L}}_A \mrm{L}\Lambda^A_{\bar{\a}}(M)  \right).
\end{align*}
Using Lemma \ref{lem:LLambda}, we have that
\begin{align*}
\flatdim_{\bar{A}} \left( \bar{A}\otimes^{\mrm{L}}_A \mrm{L}\Lambda^A_{\bar{\a}}(M)  \right) =
\flatdim_{\bar{A}} \left( \mrm{L}\Lambda^{\bar{A}}_{\bar{\a}}  ( \bar{A}\otimes^{\mrm{L}}_A M)  \right)
\end{align*}
and by Corollary \ref{cor:flat} we obtain that
\begin{align*}
\flatdim_{\bar{A}} \left( \mrm{L}\Lambda^{\bar{A}}_{\bar{\a}}  ( \bar{A}\otimes^{\mrm{L}}_A M)  \right) \le
\flatdim_{\bar{A}} \left( \bar{A}\otimes^{\mrm{L}}_A M  \right) = \flatdim_A (M)
\end{align*}
where the last equality follows from Theorem \ref{thm:flatDG}.
Combining all of the above, we obtain that
\[
\flatdim_A \left(\mrm{L}\Lambda_{\bar{\a}}(M)\right) \le \flatdim_A (M)
\]
as claimed.

\end{proof}

\begin{rem}
As in Remark \ref{exa:fails}, this result is false in general if the (DG-)ring is not assumed to be noetherian (even if the ideal is finitely generated by a regular sequence).
See \cite[Theorem 6.2]{Ye3} for an example.
\end{rem}

\begin{rem}
Let $A$ be a commutative noetherian ring
such that the Krull dimension of $A$ is $\ge 1$,
and let $B = A[x]$.
Then $B$ is a projective $B$-module,
and it follows from \cite[Theorem 1]{BF}
that the $B$-module $\Lambda_{(x)}(B)$ is not projective.
\end{rem}

We finish the paper with an important application of Theorem \ref{thm:main2}.
A basic result in commutative algebra states that if $A$ is a commutative noetherian ring,
and $\a\subseteq A$ is an ideal, then the canonical map $A \to \widehat{A}$ from $A$ to its $\a$-adic completion is flat.
Here is the analogue of this result in derived commutative algebra:

Let $A$ be a commutative noetherian DG-ring,
and let $\bar{\a}\subseteq \bar{A}$ be an ideal.
The derived $\bar{\a}$-adic completion of $A$ is a commutative DG-ring,
denoted by $\mrm{L}\Lambda(A,\bar{\a})$. It was defined in \cite[Theorem 0.2]{Sh}.
It is a derived analogue of the $\a$-adic completion $\widehat{A}$.
There is a natural map $A \to \mrm{L}\Lambda(A,\bar{\a})$,
but this is not a map of DG-rings. Instead, it only exists in a suitable homotopy category (see \cite{Sh} for details).
Concretely, one can realize it as follows: 
there is a commutative DG-ring $P$, a quasi-isomorphism $P \to A$,
and a natural map of DG-rings $P \to \mrm{L}\Lambda(A,\bar{\a})$.
Since $P$ and $A$ are quasi-isomorphic, 
the triangulated categories $\cat{D}(A)$ and $\cat{D}(P)$ are isomorphic.
Using this isomorphism and the map $P \to \mrm{L}\Lambda(A,\bar{\a})$,
we can view $\mrm{L}\Lambda(A,\bar{\a})$ as an object of $\cat{D}(A)$.
According to \cite[Proposition 3.58]{Sh}, this object is isomorphic to $\mrm{L}\Lambda_{\bar{\a}}(A)$.

The above paragraph explains that the analogue of the number $\flatdim_A(\widehat{A})$ in derived commutative algebra is the number
$\flatdim_A(\mrm{L}\Lambda_{\bar{\a}}(A))$. This explains the importance of our final result:

\begin{cor}
Let $A$ be a commutative noetherian DG-ring, and assume that $A$ has bounded cohomology.
Let $\bar{\a} \subseteq \mrm{H}^0(A)$ be an ideal.
Then
\[
\flatdim_A(\mrm{L}\Lambda_{\bar{\a}}(A)) = 0.
\]
\end{cor}
\begin{proof}
Since $\flatdim_A(A) = 0$, it follows from Theorem \ref{thm:main2} that
\[
\flatdim_A \left( \mrm{L}\Lambda_{\bar{\a}}(A) \right) \le 0. 
\]
On the other hand, by \cite[Proposition 6.1]{Sh},
we have that
\[
\mrm{H}^0\left( \mrm{L}\Lambda_{\bar{\a}}(A) \right) \cong \Lambda_{\bar{\a}}(\mrm{H}^0(A)) \ncong 0.
\]
Hence, $\sup(\mrm{L}\Lambda_{\bar{\a}}(A)) = 0$,
and we deduce that 
\[
\mrm{H}^0(\mrm{L}\Lambda_{\bar{\a}}(A) \otimes^{\mrm{L}}_A \mrm{H}^0(A)) \ne 0,
\]
which implies that 
\[
\flatdim_A \left( \mrm{L}\Lambda_{\bar{\a}}(A) \right) \ge 0. 
\]
\end{proof}

\textbf{Acknowledgments.}
The author is grateful to the anonymous referee for comments and suggestions that helped improving this manuscript.

\end{document}